\documentclass{elsarticle}
\usepackage{amsmath,amssymb}

\begin{document}
\begin{frontmatter}
\title{VC dimension of ellipsoids}
\author[ya]{Yohji Akama\corref{cor1}}
\ead{akama@m.tohoku.ac.jp}
\address[ya]
{Mathematical Institute, Tohoku University,
 Aoba-ku, Sendai, Miyagi, 980-8578, Japan\\
+81-(0)22-795-7708(tel) +81-(0)22-795-6400(fax)}
\cortext[cor1]{Corresponding author}

\author[ki]{Kei Irie}
\address[ki]{
Department of Mathematics, Kyoto University, Kyoto, 606, Japan}
\ead{iriek@math.kyoto-u.ac.jp}
\begin{abstract}
We will establish that the \textsc{vc} dimension of the class of
$d$-dimensional ellipsoids is $(d^2+3d)/2$, and that maximum likelihood
estimate with $N$-component $d$-dimensional Gaussian mixture models
induces a geometric class having \textsc{vc} dimension at least
$N(d^2+3d)/2$.
\end{abstract}
\begin{keyword} VC dimension; finite dimensional ellipsoid; Gaussian
 mixture model\end{keyword}

\end{frontmatter}
\def\evec{e} 
\def\GMM#1#2{({\mathcal G}_{#2})_{#1}}
\def\transpose#1{{}^t{#1}}
\def\Rset{\mathbb R}
\def\Sset{\mathbb S} 
\def\C{{\mathcal C}}
\def\G{{\mathcal G}_d}
\def\c#1{{\mathcal D}\left(#1\right)}
\def\P{{\mathcal P}}
\def\VC#1{\mathrm{VCdim}(#1)}
\def\diam{\mathrm{diam}\;}
\def\norm#1{\|#1\|}	
\def\conv#1{\mathrm{conv}(#1)}

\newcommand{\midc}{\;;\;}

\newtheorem{theorem}{Theorem}
\newtheorem{corollary}[theorem]{Corollary}
\newtheorem{lemma}[theorem]{Lemma}
\newproof{proof}{Proof}
\def\origin{\vec{0}}

\section{Introduction}\label{sec:intro}

For sets $X \subseteq \Rset ^d$ and $Y \subseteq X$, 
we say that a set $B \subseteq \Rset ^d$ 
\emph{cuts $Y$ out of} $X$ if 
$Y = X \cap B$. 
A class~$\mathcal C$ of subsets of $\Rset ^d$ is said to 
\emph{shatter} a set $X \subseteq \Rset ^d$ if 
every $Y \subseteq X$ is cut out of $X$ by some $B \in \mathcal C$. 
The \emph{\textsc{vc} dimension} of $\mathcal C$, denoted by $\VC{\mathcal C}$, 
is defined to be the maximum $n$ 
(or $\infty$ if no such maximum exists) for which
some subset of $\Rset ^d$ of cardinality~$n$ is shattered by $\C$. 

The \textsc{vc} dimension of a class describes a complexity of the
class, and
are employed in empirical process
theory~\cite{MR1720712}, statistical and computational learning
theory~\cite{MR1641250,MR1072253} and discrete
geometry~\cite{MR1899299}. Although 
asymptotic estimates of \textsc{vc} dimensions are given
for many classes, 
the exact values of \textsc{vc} dimensions are known for only a few
classes~(e.g. the
class of Euclidean balls~\cite{MR602047}, the class of halfspaces~\cite{MR1899299}, and so on).

In Section~\ref{sec:ellipsoids},  we prove :
\begin{theorem}\label{main} The class of $d$-dimensional ellipsoids has
 \textsc{vc} dimension
$(d^2+3d)/2$.
\end{theorem} 
Here, by a \emph{$d$-dimensional ellipsoid}, we mean an open set $\{
x\in\Rset^d \midc \transpose{(x-\mu)}A (x-\mu)< 1 \}$ where
$\mu\in\Rset^d$ and $A\in \Rset^{d\times d}$ is positive definite.

In  Section~\ref{sec:gmm}, we use a part of Theorem~\ref{main} (Lemma~\ref{thm:lowerbound}) to study
statistical models. In statistics and statistical learning theory, the class of
$d$-dimensional ellipsoids is induced from the class $\G$ of $d$-dimensional
\emph{Gaussian distributions}:  A $d$-dimensional
Gaussian distribution with mean $\mu\in\Rset^d$ and covariance matrix
$\Sigma\in\Rset^{d\times d}$ is, by definition, a probability density
function 
\begin{align*}
(2\pi)^{-d/2}
 |\det\Sigma|^{-1/2}\exp\left(-\transpose{(x-\mu)}\Sigma^{-1}(x-\mu)/2  \right),\quad(x\in\Rset^d)
\end{align*} 
where a \emph{covariance matrix} of size $d$ is, by definition, a real,
 positive definite matrix.  As in statistical learning
 theory~\cite{MR1641250}, for a class $\P$ of probability density
 functions we consider the class $\c{\P}$ of sets
$\{ x\in\Rset^d \midc  f(x) >s\}$
such that $f$ is any probability density function in $\P$ and $s$ is any
positive real number. Then $\c{\G}$ is the class of $d$-dimensional ellipsoids.

 For a positive integer $N$, 
an \emph{$N$-component
$d$-dimensional Gaussian mixture model}~\cite{MR838090}~(
 $(N,d)$-\textsc{gmm} ) is, by definition, any probability distribution
 belonging to the
convex hull of some $N$ $d$-dimensional Gaussian distributions. 
Suppose we are given  a sample from a population
$(N,d)$-\textsc{gmm} but 
the number $N$ of the components is unknown.
 To select $N$ from the sample is an example of Akaike's model selection
problem~\cite{MR0483125}~(see \cite{MR2319879} for recent approach).
The authors of \cite{wang05:_learn_gauss_mixtur_model_struc_risk_minim}
proposed to choose $N$ by \emph{structural risk minimization
principle}~\cite{MR1641250}, where an important role is played by the
\textsc{vc} dimension of the class $\c{(\G)_N}$ with $(\G)_N$ being the
class of $(N,d)$-\textsc{gmm}s.  Our result is that the \textsc{vc}
dimension of $\c{(\G)_N}$ is greater than or equal to ${N(d^2+3d)}/2.$

\section{VC dimension of ellipsoids \label{sec:ellipsoids}}

We will prove Theorem~\ref{main}.  For a positive integer
$B$, a vector $a\in \Rset^B\setminus\{ \vec{0}\}$, and $c\in\Rset$, we
write an affine function $\ell_{a,c}(x):= \transpose{a} x + c\
(x\in\Rset^B)$ and an open halfspace $H_{a,c}:=\{ x\in \Rset^B \midc
\ell_{a,c}(x)< 0\}$.  We say a set $W\subseteq \Rset^B$ \emph{spans} an
affine subspace $H\subseteq\Rset^B$, if $H$ is the smallest affine
subspace that contains $W$. The cardinality of a set $S$ is denoted by
$|S|$. For a vector $a=\transpose{(a_1,\ldots, a_B)}\in \Rset^B$, let
$\norm{a}_\infty$ be  $\max\{\,|a_i| \midc 1\le i\le B\}$.

\begin{lemma}\label{lem:abc}For any $a\in \Rset^B\setminus\{\vec{0}\}$
 and any
 $S\subset \Rset^B$ with $|S|=B$, if $S$
spans a hyperplane $\{x\in\Rset^B \midc \ell_{a,-1}(x) = 0\}$ , then $S$
 is shattered by a class  $\{ H_{b,-1} \midc   b\in \Rset^B\setminus\{\vec0\},\ \norm{b - a}_\infty <\varepsilon  \}$ for any $\varepsilon>0$.
\end{lemma}

\begin{proof}By an affine transformation we can assume without loss of
 generality that all the components of the vector $a$ are 1 and that $S$ is the
 canonical basis $\{ \evec_1,\ldots,
\evec_B\}$ of $\Rset^B$. Suppose 
 $\norm{b-a}_\infty$ is less than $\varepsilon>0$. By
 $b\ne\vec0$, we have  $H_{b,-1}\ne\Rset^B$.
Then the vector $e_i$ belongs to the open halfspace $H_{b,-1}$ if and only if the $i$-th component of $b$ is
 less than $1$.\qed
\end{proof} 

\begin{lemma} \label{thm:lowerbound}The class of $d$-dimensional
 ellipsoids has \textsc{vc} dimension greater than or equal to 
$ (d^2+3d)/2$.\end{lemma} 
\begin{proof}Let $B$ be the right-hand side.
Let $\varphi$ be a map $\Sset^{d-1} \to \Rset^B$ which maps $x=\transpose{(x_1,\ldots,x_d)}$ to $\transpose{(x_1^2,\ldots,x_d^2, x_1x_2,\ldots,x_{d-1}x_d,x_1,\ldots,x_d)}$.
Let $\transpose{(\xi_1,\ldots, \xi_B)}$ be a coordinate of
 $\Rset^B$. Then the
 image $\varphi\left(\Sset^{d-1}\right)$ spans a hyperplane
 $\xi_1+\cdots+\xi_d-1 =0$. So there is  some set
 $S\subset \Sset^{d-1}$ such that $|S|=B$ and $\varphi(S)$ spans the hyperplane. 
Let $a\in\Rset^B$ be a vector with the first $d$ components being 1 and the other
 components being 0.  By Lemma~\ref{lem:abc},
 for any  $\varepsilon>0$ the family 
$\left\{H_{b,-1} \midc   b\in\Rset^B\setminus\{\vec0\},\ \norm{b-a}_\infty < \varepsilon \right\}$
shatters $\varphi(S)$. By the definition of $\varphi$, the class of
 sets defined
 by quadratic inequalities
 \begin{align*}
  b_1 x_1^2 + \cdots + b_d x_d^2 + b_{d+1} x_1 x_2 + \cdots+ b_B x_d -1
< 0 \quad (\ \norm{b - a}_\infty <\varepsilon \ )
\end{align*}
shatters $S$. But, when $\varepsilon$ is sufficiently small, all of these sets
 are ellipsoids.\qed
\end{proof}

We verify the converse inequality.

\begin{lemma}\label{lem:hab} $\mathrm{VC}\bigl( \{H_{a,c}
 \ ;\ a=\transpose{(a_1,\ldots,a_B)}\in \Rset^B,
a_B>0,c\in\Rset\}\bigr)\le B$  for any positive integer $B$.\end{lemma}

 Below, the convex hull of a set $A$ is denoted by $\conv{A}$.

\begin{proof}Let $\C$ be $\{ H_{a,c} \midc 
 a=\transpose{(a_1,\ldots, a_B)}\in\Rset^B, a_B>0, c\in\Rset\}$. 
Assume $\VC{\C}>B$. Then $\C$ shatters
 some set $S\subset
 \Rset^B$ such that $|S|=B+1$.  

If there are $x=(u,x_B), y=(u,y_B)\in S$ such that $x_B<y_B$, then for
any $a\in \Rset^B$ with the last component nonnegative and for any
$c\in\Rset$ we have $\ell_{a,c}(x)<\ell_{a,c}(y)$, and thus $x\in
H_{a,c}=\{x\in\Rset^B \midc  \ell_{a,c}(x)<0\}$ whenever $y\in
H_{a,c}\enspace$.  This contradicts the assumption ``$\C$ shatters
$S$.'' Therefore, for the canonical projection $\pi :\Rset^B\to
\Rset^{B-1}\ ;\ (x,z)\mapsto x$, we have $|\pi(S)|=B+1$.

 By applying Radon's theorem~\footnote{Any set of $(d+2)$ points in $\Rset^d$ can be
 partitioned into two disjoint sets whose convex hulls
 intersect.}~\cite{MR1899299} to the set $\pi(S)\subset \Rset^{B-1}$, there is a partition $(T_1, T_2)$ of $S$
 such that we can take 
 $y$ from 
 $\conv{\pi(T_1)}\cap\conv{\pi(T_2)}$.  Then we see that there are $z,z'\in \Rset$
 such that $(y,z)\in \conv{T_1}$ and $(y,z')\in\conv{T_2}$. 
Because $\C$ shatters $S$, there are some $a\in \Rset^B$ and some
 $c\in\Rset$ such that the last component $a_B$ of $a$ is nonnegative and a
 halfspace $H_{a,c}\in \C$ cuts $T_1$ out of $S$. Thus, we have
$\ell_{a,c}(x) < 0$ for all $x\in \conv{T_1}$ while $\ell_{a,c}(x)\ge 0$ for
 all $x\in \conv{T_2}$ where $T_2=S\setminus T_1$.
Therefore
 $\ell_{a,c}(y,z)<\ell_{a,c}(y,z')$ and $a_B>0$, we have
 $z'>z$. 
On the other hand, some member $H_{a',c'}\in \C$ cuts $T_2$ out of $S$.
By a similar reasoning, 
 we have $z>z'$, which is a contradiction. \qed
\end{proof} 

\begin{corollary}\label{cor:hab}If $A\subset\Rset^B\setminus \{\origin\}$
and $\VC{\{H_{a,c}\}_{a\in A,c\in\Rset}}>B$, then  $\origin\in{\conv{A}}$.
\end{corollary}

\begin{proof}Let $\origin\not\in\conv A$.
Then for every
 finite subset $A'$ of $A$, $\origin\notin\conv{ A'}$ and there is a
 hyperplane $J$ through $\origin$ such that $\conv{ A'}$ is
 contained in one of the two open halfspaces determined by $J$. So
there is  a new
 rectangular coordinate system such that the origin point is the same as
 the older rectangular coordinate system,  one of the new coordinate
 axes is normal to $J$, and any $a\in A'$ is represented as $(a_1,\ldots,
 a_B)$ with $a_B>0$. 
So
 $\VC{\{H_{a,c}\}_{a\in A',c\in\Rset}}\le B$  by Lemma~\ref{lem:hab},
and thus $\VC{\{H_{a,c}\}_{a\in
 A,c\in\Rset}}\le B$.  \qed
\end{proof}

The proof of Theorem~\ref{main} is as follows:
By
Lemma~\ref{thm:lowerbound}, we have only to establish that
the class of $d$-dimensional ellipsoids has \textsc{vc} dimension less
 than or equal to 
$B:=(d^2+3d)/2$. Assume otherwise. For $a= \transpose{(a_1,\ldots,
a_B)}\in\Rset^B$ and $x=\transpose{(x_1,\ldots,x_d)}$,
define a quadratic form $q_a(x)$ and a quadratic polynomial $p_a(x)$ by
\begin{align*}
q_a(x)&:=a_1 x_1^2 +\cdots + a_d x_d^2 + a_{d+1} x_1 x_2 + \cdots +
 a_{B-d} x_{d-1} x_d,\\
p_a(x)&:=q_a(x)+a_{B-d+1}x_1 + \cdots + a_B x_d.
\end{align*}
Let
 $A$ be the set of $a \in \Rset^B$ such that $q_a$ is
positive definite. 
Obviously, $A$ is convex and $\origin \notin A$. 
Then, our assumption implies $\VC{\{H_{a,c}\}_{a\in A,c\in\Rset}}>B$, since for any ellipsoid $E$, there exists 
$a \in A$ and $c \in \Rset$ such that $E=\{ x \in \Rset^d \midc 
 p_a(x)  <-c \}$.
Hence Corollary~\ref{cor:hab} shows that $\origin \in \conv{A}=A$, which
is a contradiction. \qed

\section{A lower bound of VC dimension of
 GMMs}\label{sec:gmm}

For a positive integer $N$ and a class $\P$ of probability density
functions, let $(\P)_N$ be the class of probability density functions
$p_1f_1 + \cdots + p_N f_N$ such that $f_1,\ldots, f_N\in
\P$, $p_i\ge0$ and $p_1+\cdots +p_N=1$.
For $X\subset\Rset^d$ and $t\in\Rset^d$, put $X+t:=\{x+t \midc  x\in X\}$.
The Euclidean norm of a vector $x$ is denoted by $\norm{x}$.
Let $\diam X=\sup\{ \norm{x-x'} \midc  x,x'\in X\}$.

\begin{lemma}\label{lem:u}If a class $\P$ of probability density functions on
 $\Rset^d$ satisfies
\begin{enumerate}
\item for all  $f(x)\in \P$ and $t\in\Rset^d$ we
 have $f(x+t)\in \P$; and
\item  for any $\varepsilon>0$ there exists $a>0$ such
 that $f(x)<\varepsilon$ whenever $\norm{x}>a$,
\end{enumerate}
 then  $\VC{\c{(\P)_N}}\ge N\times \VC{\c{\P}}$.
\end{lemma}
\begin{proof}
Suppose $X\subset \Rset^d$ is shattered by $\c{\P}$. Then for each
$Y\subseteq X$ there exist
$g_Y\in\P$ and $r_Y\in\Rset$ such that
\begin{align}
Y = X\cap D_Y,\quad
D_Y=\{x\in\Rset^d \midc  g_Y(x)>e^{-r_Y}\}.\label{eq1}
\end{align}
When there is $z\in X\setminus Y$ such that $-\log
 g_Y(z)$ is equal to $r_Y$, we take a smaller $r_Y > \max\left\{ -\log g_Y(x) \midc  x\in
Y \right\}$ with the condition \eqref{eq1} kept. Then
\begin{align}
q:=\min\bigl\{  -r_{Y_j}-\log g_{Y_j}(z) \midc  z\in
 X\setminus Y_j,\ 
 Y_j\subsetneq X,\ 1\le j\le N\bigr\}  \label{q}
\end{align}
is well-defined and positive. Let $\delta>0$  be smaller than this and
all of $r_{Y_j}+\log g_{Y_j}(x)$ where $x\in
 Y_j\subseteq X$ and $1\le j\le N$.

By the assumptions (1) and (2), we
can prove that for any $j\in \{1,\ldots, N\}$, for any $\varepsilon>0$,
 for any $t_1,\ldots, t_N\in\Rset^d$
with $\norm{t_i-t_j}> \diam X$ ($i\ne j$), we have (i) 
$U:=\bigcup_{i=1}^N(X+t_i)$ has cardinality $N|X|$, and (ii) for any $x\in X$, $Y_1\subseteq
X,\ldots, Y_N\subseteq X$, for $p_{Y_i} := \exp(r_{Y_i}) / \sum_{k=1}^N
\exp(r_{Y_k})$ ($1\le i\le N$), 
\begin{align*}%
\sum_{i\ne j, 1\le i\le N} p_{Y_i}
g_{Y_i}(x-t_i)\ <\ \varepsilon
\ <\ p_{Y_j} g_{Y_j}(x-t_j) \enspace.
\end{align*}
Then the sum of the leftmost term and the rightmost term is, write
 $f(x)$,  a member of $\left(\c{\P}\right)_N$, and satisfies
$ 0< \log f(x+t_j) -  \log \left(p_{Y_j} g_{Y_j}(x)\right)
<{\varepsilon}/\left({p_{Y_j}g_{Y_j}(x)}\right)$,
since $\log (1 + u)\le u$ for any $u>0$. Because $U$ is the disjoint
union of $X+t_i$ over $1\le i \le N$, every subset $V$ of $U$ has a
 unique sequence $(Y_i)_{i=1}^N$ of subsets of $X$ such that $V=\bigcup_{i=1}^N (Y_i + t_i)$.
 So, we can define
$r_V:=\log\sum_{i=1}^N \exp (r_{Y_i})$, and $\log p_{Y_i}= r_{Y_i} - r_V$. Hence, there exist $t_1,\ldots,
t_N\in\Rset^d$ such that $\norm{t_i-t_j}> \mathrm{diam}(X)$ ($i\ne j$) and
for any $x\in X, Y_1\subseteq X,\ldots,
Y_N\subseteq X, j\in\{1,\ldots,N\}$, we have 
\begin{align}
0< \left (r_V+\log f(x+t_j)\right) -\left (r_{Y_j}+\log g_{Y_j}(x)\right)<\delta. \label{eq5}
\end{align}

Define $C_V:=\{x\in\Rset^d \midc  f(x)> \exp(-r_V)\}$, and suppose $x\in X$ and 
$j\in\{1,\ldots,N\}$. Assume $x\in Y_j$. By \eqref{eq1}, $0<r_{Y_j}+\log
g_{Y_j}(x)$. By \eqref{eq5}, $0<r_{Y_j}+\log
g_{Y_j}(x) < r_V + \log f(x+t_j)$. Therefore
$x+t_j\in C_V$. On the other hand, assume  $x\in X\setminus Y_j$.
 By \eqref{eq5} and \eqref{q}, $r_V + \log
f(x+t_j)  < r_{Y_j} + \log g_{Y_j}(x) + \delta<
r_{Y_j} + \log g_{Y_j}(x) + q\le 0$.  Thus $x+t_j\notin C_V$. To sum up, for any $x\in X$ and
any $j\in \{1,\ldots, N\}$, we have 
$x\in Y_j\iff x+t_j\in C_V$. 
Hence $U\cap C_V=V$. Thus $U$ is shattered by
$\left(\c{\P}\right)_N$. \qed
\end{proof} 

By Lemma~\ref{thm:lowerbound} and Lemma~\ref{lem:u}, we have:
\begin{corollary} The \textsc{vc} dimension of $(N,d)$-\textsc{gmm}s is greater than or equal to
$N(d^2+3d)/2$. In other words, for
the class $(\G)_N$ of $(N,d)$-\textsc{gmm}s, the
 class $\c{(\G)_N}$ has the \textsc{vc} dimension greater than or equal to
$N(d^2+3d)/2$.\end{corollary}

\section{Conclusion} We can easily obtain an asymptotically tight estimate of the class of
$d$-dimensional ellipsoids through the combination of a naive
linearization argument~\cite{MR1899299} and an approximation argument of ``affine
subspaces''~(\emph{bands}~\cite{springerlink:10.1007/s00454-009-9236-5},
more precisely) by  ellipsoids. However, we in Section~\ref{sec:ellipsoids} have provided the \emph{exact} value of the \textsc{vc}
dimension, by combining a linearization argument~\cite{MR1899299,MR602047} with an argument about convex
bodies.
Our argument seems useful to establish the \textsc{vc} dimension of the
class of bounded sets $\{x\in\Rset^d \midc   \mbox{$p(x)>0$}\}$ such that
$p$ is any real polynomial with bounded degree. 


\section*{Acknowledgements}The first author is partially supported by Grant-in-Aid for Scientific
Research (C) (21540105) of the Ministry of Education, Culture, Sports,
Science and Technology (MEXT). The second author is Supported by
Grant-in-Aid for JSPS Fellows.

\end{document}